\documentclass[12 pt]{article}
\usepackage{rotating}
\usepackage{times}
\usepackage{booktabs}
\usepackage{pifont}
\usepackage{floatrow}
\usepackage{caption}
\usepackage{mathrsfs}
\usepackage{longtable}
\usepackage[fleqn]{amsmath}
\usepackage{amsfonts,amsthm,amssymb,mathrsfs,bbding}
\usepackage{txfonts}
\usepackage{graphics,multicol}
\usepackage{graphicx}
\usepackage{color}
\usepackage{caption}
\usepackage{indentfirst}
\usepackage{cite}
\usepackage{latexsym,bm}
\evensidemargin=\oddsidemargin\topmargin=-1.5cm

\newtheorem{lem}{Lemma}[section]
\newtheorem{defi}{Definition}[section]
\newtheorem{thm}[lem]{Theorem}
\newtheorem{cor}[lem]{Corollary}

\newtheorem{remark}{Remark}

\newcommand{\tabincell}[2]{\begin{tabular}{@{}#1@{}}#2\end{tabular}}
\theoremstyle{definition}

\addtocounter{section}{0}

\begin{document}
\title{The graphs with exactly two distance eigenvalues different from $-1$ and $-3$\footnote{Supported
by the National Natural Science Foundation of China (Grant Nos. 11261059, 11531011).}}
\author{{\small Lu Lu, \ \ Qiongxiang Huang,\footnote{
Corresponding author.
Email: huangqx@xju.edu.cn}\ \ Xueyi Huang}\\[2mm]\scriptsize
College of Mathematics and Systems Science,
\scriptsize Xinjiang University, Urumqi, Xinjiang 830046, P.R.China}

\date{}
\maketitle {\flushleft\large\bf Abstract} In this paper, we completely characterize the graphs with third largest distance eigenvalue at most $-1$ and smallest distance eigenvalue at least $-3$. In particular, we determine all graphs whose distance matrices have exactly two eigenvalues (counting multiplicity) different from $-1$ and $-3$. It turns out that such graphs consist of three infinite classes, and all of them are determined by their distance spectra. We also show that the friendship graph is determined by its distance spectrum.

\vspace{0.1cm}
\begin{flushleft}
\textbf{Keywords:} distance eigenvalue; distance equitable partition; friendship graph; distance spectral characterization
\end{flushleft}
\textbf{AMS Classification:} 05C50
\section{Introduction}\label{s-1}
Let $G$ be a simple connected graph with vertex set $V(G)=\{v_1,v_2,\ldots,v_n\}$, edge set $E(G)$ and adjacency matrix $A=A(G)$. Denote by $d(v_i,v_j)$  the  \emph{distance} (i.e., the length of a shortest path) between the vertices $v_i$ and $v_j$ of $G$. Then the \emph{diameter} $d(G)$ and \emph{distance matrix} $D(G)$ are defined as $d(G)=\max\{d(v_i,v_j)\mid v_i,v_j\in V(G)\}$ and $D(G)=(d(v_i,v_j))_{n\times n}$, respectively.

Since $D(G)$ is a real symmetric matrix, its all eigenvalues are real and can be conventionally denoted and arranged as $\partial_1\ge\partial_2\ge\cdots\ge\partial_n$.  These eigenvalues are also called the  \emph{distance eigenvalues} of $G$, and the largest one $\partial_1$ is called the \emph{distance spectral radius} of $G$. The \emph{distance spectrum} of the graph $G$, denoted by $\mathrm{Spec}_D(G)$, consists of the  distance eigenvalues  (together with their multiplicities). The graph $G$ is said to be \textit{determined by its distance spectrum} (DDS for short) if, for any graph $H$,  $\mathrm{Spec}_D(H)=\mathrm{Spec}_D(G)$ implies that $H\cong G$. The notions of \emph{adjacency eigenvalue}, \emph{adjacency spectrum} (denoted by $\mathrm{Spec}_A(G)$) and \emph{determined by its adjacency spectrum} (DAS for short) can be similarly defined if we consider the adjacency matrix $A(G)$.

Throughout this paper, we denote by $J_{i\times j}$ the $i\times j$ all-ones matrix, $I_p$ the identity matrix of order $p$ and $N(v)=\{u \in V(G)\mid uv\in E(G)\}$ the \emph{neighbourhood} of  $v\in V(G)$. The \emph{join} of two vertex disjoint graphs $G$ and $H$ is the graph $G\vee H$ obtained by adding all edges with one end vertex in $G$ and the other end vertex in $H$. In addition, all the symbols and notions not mentioned here are standard and can be found in \cite{Dragos}.

The famous question ``Which graphs are determined by their spectra?'' has been raised by G\"{u}nthard and Primas \cite{Gunthard} over half a century and is far from being solved up to now. For surveys of this question see \cite{Dam1,Dam2}. With regard to distance spectrum, it is believed that a mass of graphs are DDS since the distance matrix $D(G)$ contains more information than $A(G)$. However,  just a few of DDS-graphs  are characterized up to now.  Recently, Jin and Zhang \cite{Jin} proved that the complete $k$-partite graph $K_{n_1,\cdots,n_k}$ is DDS; Lin, Zhai and Gong \cite{Huiqiu2} proved that the graph $K_{s,t}^r=K_r\vee (K_s\cup K_t)$ with $r\ge 1$ is DDS.

The famous friendship graph $F_k$ consists of $k$
edge disjoint triangles that all of them meeting in one vertex. In 2010, Wang et al. \cite{Wang} put forward the conjecture that $F_k$ is DAS. This conjecture aroused several activities \cite{Das,Abdollahi} and finally was affirmed by Cioab\u{a} et al. \cite{Sebastian} for $k\neq 16$ (if $k=16$, they also showed that there is  exactly one graph $H$ satisfying $\mathrm{Spec}_A(H)=\mathrm{Spec}_A(F_{k})$ but $H\not\cong F_k$). Actually, Cioab\u{a} et al. characterized all graphs with all but two adjacency eigenvalues equal to $\pm1$ and $F_k$ is just contained in this class.

In this paper,  we first introduce the notion of distance equitable partition and give some basic results about it in section \ref{s-2}. Motivated by the work of Cioab\u{a} et al., we completely characterize those graphs satisfying $-3\le\partial_{n}$ and $\partial_3\le-1$ in section \ref{s-3}. In particular, we determine all graphs with exactly two distance eigenvalues (counting multiplicity) different from $-1$ and $-3$ in section \ref{s-4}. In the process, we show that all these graphs are DDS, and particularly, $F_k$ is DDS.

\section{The distance equitable partition}\label{s-2}
Given a graph $G$, the vertex partition $\Pi$: $V(G)=V_1\cup V_2\cup\cdots\cup V_r$ is said to be an \emph{equitable partition} if, for any $u\in V_i$, $|V_j\cap N(u)|=b_{ij}$ is a constant only dependent on $i,j$ ($1\le i,j\le r$), and the matrix $B_\Pi=(b_{ij})_{k\times k}$ is the so-called \emph{divisor matrix} derived from $\Pi$. Moreover, $B_\Pi$ can be regarded as the adjacency matrix of a directed multigraph $G/\Pi$, which is called the \emph{divisor} of $G$ with respect to $\Pi$. Let $A$ be the  adjacency matrix of $G$, and let $C$ be the \textit{characteristic matrix} with respect to $\Pi$ whose $i$-th column $C_i$ is the characteristic function of $V_i$. Then $AC=CB_\Pi$, and so the columns of $C$ generate an invariant subspace of $A$, which produces  a nice  property: $\det(xI-B_\Pi)|\det(xI-A)$ (see \cite{Dragos}, Theorem 3.9.5). Naturally, we ask if there exists  analogous ``equitable partition'' for the distance matrix of $G$? If it exists, what confuses us is that how it reveals the relation between the eigenvalues and the structure of a graph. In this section, we will introduce the notion of ``equitable partition'' for the distance matrix of  a graph.

Denote by $d(v, S)=\sum_{u\in S}d(u,v)$, where $v\in V(G)$ and $S$ is a nonempty subset of $V(G)$. In terms of $d(v,S)$, we give the following definition of distance equitable partition.

\begin{defi}\emph{Given a connected graph $G$, the vertex partition $\Pi:$ $V(G)=V_1\cup V_2\cup\cdots\cup V_k$ is called a \emph{distance equitable partition} if, for any $v\in V_i$, $d(v,V_j)=b_{ij}$ is a constant only dependent on $i,j$ ($1\le i,j\le k$). Here the matrix $B^*_\Pi=(b_{ij})_{k\times k}$ is called the \emph{distance divisor matrix} of $G$ with respect to $\Pi$.}
\end{defi}

Let $\Gamma$ be the automorphism group of $G$, and $O_1,O_2,\ldots,O_k$ all the orbits of $\Gamma$. Then $\Pi$: $V(G)=O_1\cup O_2\cup\cdots\cup O_k$ is a partition of $V(G)$ which is generally called the \emph{orbit partition} of $G$. It is well known that the orbit partition is an equitable partition and,  fortunately, we have the  following result.

\begin{lem}\label{lem-2-1}
Let $G$ be a connected graph. Then its orbit partition $\Pi:$ $V(G)=O_1\cup O_2\cup\cdots\cup O_k$ is also a distance equitable partition.
\end{lem}
\begin{proof}
Notice that $d(u,v)=d(\sigma(u),\sigma(v))$ for any automorphism $\sigma$ of $G$. For  $v\in O_i$, we have
\begin{equation}\label{D-eq-2}
d(\sigma(v),O_j)=d(\sigma(v),\sigma(O_j))=\sum\limits_{u\in O_j}d(\sigma(v),\sigma(u))=\sum\limits_{u\in O_j}d(v,u)=d(v,O_j).
\end{equation}
Since the automorphism group $\Gamma$ acts transitively on each orbit, $d(v,O_j)$ is a constant independent on the choice of $v\in O_i$. The result follows.
\end{proof}

Now suppose that $\Pi$: $V(G)=V_1\cup V_2\cup\cdots \cup V_k$ is a distance equitable partition of $G$, and $C$ is the characteristic matrix with respect to $\Pi$. Then the distance divisor matrix $B_\Pi^*=(b_{ij})$ (with respect to $\Pi$) leads to a \emph{distance divisor} of $G$, also denoted by $G/\Pi$, which is the directed multigraph with vertices $V_1,V_2,\ldots,V_k$  and $b_{ij}$ arcs from $V_i$ to $V_j$. To compare with (adjacency) equitable partition, we will give some parallel  results for distance equitable partition.

\begin{lem}\label{lem-2-2}
Let $G$ be a connected graph with a distance equitable partition $\Pi:$ $V(G)=V_1\cup V_2\cup\cdots \cup V_k$. Then $DC=CB_\Pi^*$, where $D$ is the distance matrix of $G$, $C$ and $B_\Pi^*$ are the characteristic matrix and distance divisor matrix with respect to $\Pi$, respectively.
\end{lem}
\begin{proof}
Consider the $(v,j)$-entry of $DC$ and $CB_{\Pi}^*$ where $v$ is in $V_i$. On the one hand,
$$(DC)_{vj}=\sum_{u\in V_j}d(v,u)=d(v,V_j)=b_{ij}.$$
On the other hand, $(CB_{\Pi}^*)_{vj}=b_{ij}$. Thus our result follows.
\end{proof}
\begin{thm}\label{thm-2-3}
Let $G$ be a connected graph with distance matrix $D$, and let $\Pi:$ $V=V_1\cup V_2\cup\cdots\cup V_k$ be a distance equitable partition of $G$ with distance divisor matrix $B_{\Pi}^*$. Then $\det(xI-B_{\Pi}^*)|\det(xI-D)$.
\end{thm}
\begin{proof}
Let $C$ be the characteristic matrix with respect to $\Pi$. Clearly, the matrix $C$ has rank $k$. Now we choose a  matrix $C^*$ of order $n\times (n-k)$ such that $(C\mid C^*)$ is an invertible matrix of order $n\times n$, where $n=|V(G)|$. Then there exist two matrices $X$ and $Y$ such that
\begin{equation}\label{Lem-2-1}
DC^*=CX+C^*Y.
\end{equation}
From Lemma \ref{lem-2-2} and (\ref{Lem-2-1}), we obtain
\begin{equation*}
D(C\mid C^*)=(C\mid C^*)\left(
                            \begin{array}{cc}
                              B_{\Pi}^* & X \\
                              0 & Y \\
                            \end{array}
                          \right).
\end{equation*}
It follows that $\det(xI-D)=\det(xI-B_{\Pi}^*)\det(xI-Y)$ since $(C\mid C^*)$ is invertible.
\end{proof}

\begin{cor}\label{cor-2-4}
Let $G$ be a connected graph of order $n$ with distance equitable partition $\Pi$, and $B_{\Pi}^*$ the distance divisor matrix of $G$ with respect to $\Pi$. Then the largest eigenvalue of $B_{\Pi}^*$ is just the distance spectral radius of $G$.
\end{cor}
\begin{proof}
Let $\lambda$ be the largest eigenvalue of $B_{\Pi}^*$ with eigenvector $\mathbf{x}$, then
$B_{\Pi}^*\mathbf{x}=\lambda\mathbf{x}$. By Perron-Frobenius Theorem, we may assume that $\mathbf{x}>0$. Putting $\mathbf{y}=C\mathbf{x}$, where $C$ is the characteristic matrix with respect to $\Pi$. From Lemma \ref{lem-2-2}, we have
$$D\mathbf{y}=D(C\mathbf{x})=(DC)\mathbf{x}=(CB_{\Pi}^*)\mathbf{x}=C(B_{\Pi}^*\mathbf{x})=C(\lambda\mathbf{x})=\lambda(C\mathbf{x})=\lambda\mathbf{y}.$$
Thus $\mathbf{y}$ is an eigenvector of $D$ and, by Perron-Frobenius Theorem again, $\lambda$ is the distance spectral radius of $G$ because $\mathbf{y}$ is positive.
\end{proof}

\section{The graphs with distance spectrum $[\partial_n,\partial_3]\subseteq[-3,-1]$}\label{s-3}
There are some results about the smallest distance eigenvalue $\partial_n$ of a graph. Recently, Yu \cite{Guanglong} proved that $\partial_n(G)\le-2.383$ when $G$ is neither a complete graph nor a complete $k$-partite graph. In this section, we will  characterize those graphs satisfying $-3\le\partial_n$ and $\partial_3\le-1$.
\begin{lem}\label{lem-3-1}
(Cauchy Interlace Theorem). Let $A$ be a Hermitian matrix with order $n$, and
let $B$ be a principal submatrix of $A$ with order $m$. If $\lambda_1(A)\ge\lambda_2(A)\ge\cdots\ge\lambda_n(A)$ lists the eigenvalues of $A$ and $\mu_1(B)\ge\mu_2(B)\ge\cdots\ge\mu_m(B)$ the eigenvalues of $B$, then
$$\lambda_{n-m+i}(A)\le\mu_i(B)\le\lambda_i(A)\textrm{ for }i=1,\ldots,m.$$
\end{lem}
Let $H$ be a connected induced subgraph of the connected graph $G$. Then $A(H)$ must be a principle submatrix of $A(G)$, while $D(H)$ may not be a principal submatrix of $D(G)$. For example, the path $P_4$ is an induced subgraph of the cycle $C_5$, and the distance matrix of $P_4$ and $C_5$ are respectively given by
$$
D(P_4)=\left(
       \begin{array}{cccc}
         0 & 1 & 2 & 3 \\
         1 & 0 & 1 & 2 \\
         2 & 1 & 0 & 1 \\
         3 & 2 & 1 & 0 \\
       \end{array}
     \right),
D(C_5)=\left(
         \begin{array}{ccccc}
           0 & 1 & 2 & 2 & 1 \\
           1 & 0 & 1 & 2 & 2 \\
           2 & 1 & 0 & 1 & 2 \\
           2 & 2 & 1 & 0 & 1 \\
           1 & 2 & 2 & 1 & 0 \\
         \end{array}
       \right).
$$
Obviously, $D(P_4)$ is not a principal submatrix of $D(C_5)$. Nevertheless, if $H$ is an induced subgraph of $G$ with diameter $d(H)<3$, $D(H)$ must be a principal submatrix of $D(G)$. In fact, for any $u,v\in H$, $d_H(u,v)\le d(H)<3$. This implies that either $u$ and $v$ are adjacent or they are not adjacent but have a common neighbor in $H$. If $u$ and $v$ are adjacent in $H$, then they are adjacent in $G$, and so $d_H(u,v)=d_G(u,v)=1$. If $u$ and $v$ are not adjacent but have a common neighbor in $H$, then they are not adjacent in $G$ and still have a common neighbor because $H$ is an induced subgraph of $G$, and so $d_H(u,v)=d_G(u,v)=2$. Thus for any $u,v\in H$, we have $d_H(u,v)=d_G(u,v)$, which implies that $D(H)$ is a principal submatrix of $D(G)$. By Lemma \ref{lem-3-1}, we have the following result.
\begin{lem}\label{lem-3-2}
If $H$ is a connected induced subgraph of $G$ with diameter $d(H)<3$, then the eigenvalues of $D(H)$ interlace that of $D(G)$.
\end{lem}

\begin{cor}\label{cor-3-3}
Let $G_1$ be a connected graph with diameter $d(G_1)<3$. If $G=K_r\vee G_1$, then the eigenvalues of $D(G_1)$ interlace that of $D(G)$.
\end{cor}
The distance eigenvalues are closely linked to the structure of a graph. In fact, some special structure of a graph can lead to some special distance eigenvalues. Conversely, some special distance eigenvalues also can determine some special structure of a graph.
\begin{lem}\label{lem-3-4}
Let $G$ be a connected graph on $n$ vertices. If $S=\{v_1,\ldots,v_p\}$ $(p\ge2)$ induces a clique of $G$ with $N(v_i)\setminus S=N(v_j)\setminus S$ for $1\le i,j\le p$. Then $-1$ is an eigenvalue of $D(G)$ with multiplicity at least $p-1$.
\end{lem}
\begin{proof}
According to our assumption, we can suppose that  $N= N(v_i)\setminus S=\{u_1,u_2,\ldots,u_q\}$ where  $i=1,2,\ldots,p$.  Set $T=V(G)\setminus(S\cup N)=\{w_1,w_2,\ldots,w_{n-p-q}\}$. Then $V(G)=S\cup N\cup T$ is a vertex partition of $G$.  Since $d(v_i,u_j)=1$, the submatrix $D(S,N)$ of $D(G)$ induced on the row set $S$ and the column set $N$ equals to $J_{p\times q}$. Similarly, since $d(v_i,v_j)=1$, we have $D(S,S)=J_{p\times p}-I_p$. Furthermore, for $j\in \{1,2,\dots,n-p-q\}$ we see that $d(v_i,w_j)=a_j$ for any $i\in \{1,2,\dots,p\}$. Hence the submatrix $D(S,T)$ can be written as
$$D(S,T)=\begin{array}{c@{}c@{}c@{}c@{}c@{}c@{}c@{}}
     w_1 & w_2 & \cdots&w_{n-p-q}&  \\
                     \left(\begin{array}{c} a_1 \\ a_1\\\vdots \\ a_1 \end{array}\right.
                    & \begin{array}{c} a_2\\a_2\\\vdots \\a_2 \end{array}
                    & \begin{array}{c} \cdots \\ \cdots \\ \ddots\\\cdots \end{array}
                    & \begin{array}{c} a_{n-p-q}\\a_{n-p-q}\\\vdots \\a_{n-p-q} \end{array}
                    & \left)\begin{array}{c} v_1 \\v_2 \\ \vdots \\v_p \end{array}\right.
  \end{array}$$
Now putting $A=D(S,T)$, $X=D(N,N)$, $Y=D(N,T)$ and $Z=D(T,T)$, the distance matrix $D(G)$ can be written as
$$D(G)=\begin{array}{c@{}c@{}c@{}c@{}c@{}c@{}}
     S & N & T & \\
    \left(\begin{array}{c} J_{p\times p}-I_p \\ J_{q\times p} \\ A^T \end{array}\right.
                    & \begin{array}{c} J_{p\times q} \\X \\ Y^T\end{array}
                         & \begin{array}{c} A \\ Y \\ Z \end{array}
                          & \left)\begin{array}{c} S\\ N\\T\end{array}\right.
  \end{array}.$$
For $i=2,\ldots,p$, let $x^{(i)}\in \mathbb{R}^n$ be the vector defined on $V(G)$ with $x^{(i)}_{v_1}=1$, $x^{(i)}_{v_i}=-1$ and $x^{(i)}_v=0$ for $v\ne v_1, v_i$. Since $A^T(x^{(i)}_{v_1},\dots,x^{(i)}_{v_p})^T=0$, we have $D(G)x^{(i)}=(-1)x^{(i)}$. Moreover, $x^{(2)},x^{(3)},\ldots,x^{(p)}$ are linearly independent. Thus the result follows.
\end{proof}
\begin{lem}\label{lem-3-5}
Let $G$ be a connected graph on $n$ vertices. If $S=mK_r$ $(m\ge2)$ is an induced subgraph of $G$ with $N(u)\setminus V(S)=N(v)\setminus V(S)$ for any $u,v\in V(S)$, then $-(r+1)$ is an eigenvalue of $D(G)$ with multiplicity at least $m-1$.
\end{lem}
\begin{proof}
First we partition the vertices of $S$ as $V(S)=V(S_1)\cup V(S_2)\cup\cdots\cup V(S_m)$ such that each $V(S_i)$ induces a $K_r$ in $G$, where $i=1,\ldots,m$. Assume that $V(S_i)=\{v^{(i)}_1,\ldots,v^{(i)}_r\}$, $N=N(v)\setminus V(S)=\{u_1,\ldots,u_q\}$ for any $v\in S$ and $T=V(G)\setminus(N\cup V(S))=\{w_1,\ldots,w_{n-q-mr}\}$, like the proof of Lemma \ref{lem-3-4}, the distance matrix $D(G)$ can be written as
$$D(G)=\begin{array}{@{}r@{}c@{}c@{}c@{}l@{}}
    V(S) & N & T & \\
               \left( \begin{array}{c} B \\ J_{q\times mr} \\ A^T \end{array}\right.
                    & \begin{array}{c} J_{mr\times q} \\X \\ Y^T\end{array}
                         & \begin{array}{c} A \\ Y \\ Z \end{array}
                          & \left)\begin{array}{c} V(S)\\ N\\T  \end{array}\right.
  \end{array}$$
where
$$
B=\begin{array}{c@{}c@{}c@{}c@{}c@{}c@{}c@{}}
     V(S_1) & V(S_2) & \cdots & V(S_m)& \\
    \left(\begin{array}{c} J_{r\times r}-I_r \\ 2J_{r\times r} \\ \vdots\\ 2J_{r\times r} \end{array}\right.
                    & \begin{array}{c} 2J_{r\times r} \\ J_{r\times r}-I_r \\ \vdots \\ 2J_{r\times r} \end{array}
                         & \begin{array}{c} \cdots \\ \cdots \\ \ddots\\ \cdots \end{array}
                         & \begin{array}{c} 2J_{r\times r} \\ 2J_{r\times r} \\ \vdots \\ J_{r\times r}-I_r \end{array}
                          & \left)\begin{array}{c} V(S_1)\\ V(S_2)\\ \vdots \\ V(S_m) \end{array}\right.
  \end{array}, A=\begin{array}{c@{}c@{}c@{}c@{}c@{}c@{}c@{}}
     w_1 & w_2 & \cdots&w_{n-q-mr} & \\
     \left(\begin{array}{c} a_1 \\ a_1\\\vdots \\ a_1 \end{array}\right.
                    & \begin{array}{c} a_2\\a_2\\\vdots \\a_2 \end{array}
                    & \begin{array}{c} \cdots \\ \cdots \\ \ddots\\\cdots \end{array}
                    & \begin{array}{c} a_{n-q-mr}\\a_{n-q-mr}\\\vdots \\a_{n-q-mr} \end{array}
                    & \left)\begin{array}{c}v^{(1)}_1 \\v^{(1)}_2  \\ \vdots\\ v^{(m)}_r  \end{array}\right.
  \end{array}
$$

For $i=2,\ldots,m$, let $y^{(i)}\in \mathbb{R}^n$ be the vector defined on $V(G)$ with $y^{(i)}_{v^{(1)}_j}=1$, $y^{(i)}_{v^{(i)}_j}=-1$ for $j=1,\ldots,r$ and $y^{(i)}_v=0$ for $v\not\in\{v^{(1)}_j,v^{(i)}_j\mid 1\le j\le r\}$. Since $B(y^{(i)}_{v^{(1)}_1},\dots,y^{(i)}_{v^{(m)}_r})^T=-(r+1)(y^{(i)}_{v^{(1)}_1},\dots,y^{(i)}_{v^{(m)}_r})^T$ and $A^T(y^{(i)}_{v^{(1)}_1},\dots,y^{(i)}_{v^{(m)}_r})^T=0$, we have $D(G)y^{(i)}=-(r+1)y^{(i)}$. Moreover, $y^{(2)},\dots,y^{(m)}$ are linear independent. Thus the result follows.
\end{proof}

If $S=\{v_1,\ldots,v_p\}$ $(p\ge2)$ is an independent set of the connected graph $G$ with $N(v_i)=N(v_j)$ for $1\le i,j\le p$, we say that $S$ is a \emph{star independent set} of  order $p$. The following two corollaries are special cases of Lemma \ref{lem-3-5} for  $r=1$ and $r=2$, respectively.

\begin{cor}\label{cor-3-6}
Let $G$ be a connected graph. If $G$ contains a star independent set of order $p$, then $-2$ is an eigenvalue of $D(G)$ with multiplicity at least $p-1$.
\end{cor}

\begin{cor}\label{cor-3-7}
Let $G$ be a connected graph. If $S=mK_2$ $(m\ge2)$ is an induced subgraph of $G$ with $N(u)\setminus V(S)=N(v)\setminus V(S)$ for any $u,v\in V(S)$, then $-3$ is an eigenvalue of $D(G)$ with multiplicity at least $m-1$.
\end{cor}
Let $G$ be a graph with vertex set $V(G)$. For any $X\subseteq V(G)$, we say that $X$ is \emph{$G$-connected if the subgraph $G[X]$ of $G$ induced by $X$ is connected}.
\begin{lem}[\cite{Seinsche}]\label{lem-3-8}
Let $G$ be a graph. The following statements are equivalent:\\
(1) $G$ has no induced subgraph isomorphic to $P_4$.\\
(2) Every subset of $V(G)$ with more than one element is not $G$-connected or not $\bar{G}$-connected.
\end{lem}
Let $G$ be a connected graph containing no induced $P_4$. Then $V(G)$ is a subset of itself and so is $G$-connected, by Lemma \ref{lem-3-8}, we know that $\bar{G}$ is disconnected. Therefore, we get the following result.
\begin{lem}\label{lem-3-9}
If $G\not=K_r$ is a connected graph containing no induced $P_4$, then $G$ must be a join of two graphs, i.e., $G\cong G_1\vee G_2$, where $G_1$ and $G_2$ are non-null.
\end{lem}

From Lemma \ref{lem-3-9}, we know that the diameter of a non-complete connected graph containing no induced $P_4$ is two. However, a graph with diameter two may contain induced $P_4$ such as the cycle $C_5$.

Denote by $\mathcal{G}[-3\le\partial_n,\partial_3\le-1]$ the set of non-complete connected graphs of order $n$ ($n\ge4$) satisfying  $-3\le\partial_n(G)$ and $\partial_3(G)\le-1$. In the following, we try to characterize the graphs in $\mathcal{G}[-3\le\partial_n,\partial_3\le-1]$. We start with a list of forbidden induced subgraphs shown in Fig.\ref{fig-1}.
\begin{lem}\label{lem-3-10}
No graph in $\mathcal{G}[-3\le\partial_n,\partial_3\le-1]$ has one of the graphs $P_4$, $C_5$, $H_i$ ($i=0,1,\ldots,6$) and $I_j$ ($j=1,2,3,4$) (shown in Fig. \ref{fig-1}) as an induced subgraph.
\end{lem}
\begin{proof}
Let $G\in\mathcal{G}[-3\le\partial_n,\partial_3\le-1]$.
Since the diameters of $C_5$, $H_i$ ($i=0,1,\ldots,6$) and $I_j$ ($j=1,2,3,4$) are all less than $3$, and each of these graphs has its third largest distance eigenvalue $\partial_3$ strictly greater than $-1$ or its smallest distance eigenvalue $\partial_n$ strictly less than $-3$, by Lemma \ref{lem-3-2}, none of $C_5$, $H_i$ $(i=0,\ldots,6)$ and $I_j$ ($j=1,2,3,4$) can be an induced subgraph of $G$. In the following, it suffices to show that $P_4$ cannot be an induced subgraph of $G$.

By contradiction, assume that $G$ contains induced $P_4=v_1v_2v_3v_4$. If $N(v_1)\cap N(v_4)=\emptyset$, then $D(P_4)$ is a principal submatrix of $D(G)$. By Lemma \ref{lem-3-1}, $-3\le\partial_n(G)\le\lambda_4(D(P_4))=-3.14$, a contradiction. Next, we assume that there exists $v\in N(v_1)\cap N(v_4)$. If $v\not\sim v_2$ and $v\not\sim v_3$, then $C_5$ will be an induced subgraph of $G$, a contradiction. If $v\sim v_2$ and $v\not\sim v_3$ (see $H_0$ in Fig.\ref{fig-1}), or $v\not\sim v_2$ and $v\sim v_3$, then $H_0$ will be an induced subgraph of $G$, a contradiction. If $v\sim v_2$ and $v\sim v_3$ (see $H_1$ in Fig.\ref{fig-1}), then $H_1$ will be an induced subgraph of $G$, a contradiction. Now we complete the proof.
\end{proof}

\begin{figure}[t]
  \centering
\unitlength 3.7mm 
\linethickness{0.4pt}
\ifx\plotpoint\undefined\newsavebox{\plotpoint}\fi %
\begin{picture}(33.5,22.25)(0,0)
\thicklines
\put(.25,21){\line(0,-1){3}}
\put(.25,18){\line(1,0){6}}
\put(6.25,18){\line(0,1){3}}
\put(9.25,21){\line(0,-1){3}}
\put(9.25,18){\line(1,0){6}}
\put(15.25,18){\line(0,1){3}}
\multiput(15.25,21)(-.1,.0333333){30}{\line(-1,0){.1}}
\multiput(12.25,22)(-.1,-.0333333){30}{\line(-1,0){.1}}
\put(18.25,22){\line(1,0){6}}
\put(24.25,22){\line(-3,-4){3}}
\put(21.25,18){\line(-3,4){3}}
\multiput(20.25,22)(.0333333,-.1333333){30}{\line(0,-1){.1333333}}
\put(27.25,22){\line(1,0){6}}
\put(33.25,22){\line(-3,-4){3}}
\put(30.25,18){\line(-3,4){3}}
\multiput(29.25,22)(.0333333,-.1333333){30}{\line(0,-1){.1333333}}
\multiput(30.25,18)(.0333333,.1333333){30}{\line(0,1){.1333333}}
\put(.25,14){\line(1,0){6}}
\put(6.25,14){\line(0,-1){4}}
\put(6.25,10){\line(-1,0){6}}
\put(.25,10){\line(0,1){4}}
\put(.25,14){\line(3,-2){6}}
\put(6.25,14){\line(-3,-2){6}}
\put(.25,10){\line(3,4){3}}
\put(3.25,14){\line(0,-1){4}}
\put(3.25,10){\line(3,4){3}}
\put(3.25,10){\line(-3,4){3}}
\put(3.25,14){\line(3,-4){3}}
\put(9.25,14){\line(1,0){4}}
\multiput(13.25,14)(-.03333333,-.06666667){60}{\line(0,-1){.06666667}}
\multiput(11.25,10)(-.03333333,.06666667){60}{\line(0,1){.06666667}}
\put(11.25,10){\line(1,1){4}}
\put(11.25,14){\line(0,-1){4}}
\put(18.25,14){\line(1,0){6}}
\put(24.25,14){\line(0,-1){4}}
\put(24.25,10){\line(-3,4){3}}
\put(21.25,14){\line(-3,-4){3}}
\put(18.25,10){\line(0,1){4}}
\put(18.25,14){\line(3,-2){6}}
\put(24.25,14){\line(-3,-2){6}}
\put(27.25,14){\line(1,0){3}}
\multiput(30.25,14)(-.0333333,-.0333333){30}{\line(0,-1){.0333333}}
\multiput(29.25,13)(-.0666667,.0333333){30}{\line(-1,0){.0666667}}
\multiput(27.25,14)(.03333333,-.06666667){60}{\line(0,-1){.06666667}}
\put(29.25,10){\line(0,1){3}}
\multiput(30.25,14)(-.0333333,-.1333333){30}{\line(0,-1){.1333333}}
\multiput(29.25,10)(.03333333,.06666667){60}{\line(0,1){.06666667}}
\put(31.25,14){\line(1,0){2}}
\put(33.25,14){\line(-1,-1){4}}
\multiput(3.25,6)(-.1,-.0333333){30}{\line(-1,0){.1}}
\put(.25,5){\line(0,-1){2}}
\multiput(.25,3)(.0666667,-.0333333){30}{\line(1,0){.0666667}}
\put(2.25,2){\line(1,0){2}}
\multiput(4.25,2)(.0666667,.0333333){30}{\line(1,0){.0666667}}
\put(6.25,3){\line(0,1){2}}
\multiput(6.25,5)(-.1,.0333333){30}{\line(-1,0){.1}}
\put(3.25,6){\line(1,0){3}}
\multiput(3.25,6)(-.03370787,-.03370787){89}{\line(0,-1){.03370787}}
\put(.25,3){\line(1,0){6}}
\multiput(6.25,3)(-.03370787,.03370787){89}{\line(0,1){.03370787}}
\put(.25,5){\line(1,0){6}}
\put(6.25,5){\line(-3,-1){6}}
\put(.25,5){\line(3,-1){6}}
\multiput(.25,5)(.03333333,-.05){60}{\line(0,-1){.05}}
\multiput(2.25,2)(.0333333,.1333333){30}{\line(0,1){.1333333}}
\multiput(3.25,6)(.0333333,-.1333333){30}{\line(0,-1){.1333333}}
\multiput(4.25,2)(.03333333,.05){60}{\line(0,1){.05}}
\put(6.25,5){\line(-4,-3){4}}
\put(.25,5){\line(4,-3){4}}
\multiput(.25,3)(.1333333,-.0333333){30}{\line(1,0){.1333333}}
\multiput(6.25,3)(-.1333333,-.0333333){30}{\line(-1,0){.1333333}}
\multiput(11.25,6)(-.0666667,-.0333333){30}{\line(-1,0){.0666667}}
\put(9.25,5){\line(0,-1){3}}
\put(9.25,2){\line(1,0){4}}
\put(13.25,2){\line(0,1){3}}
\multiput(13.25,5)(-.0666667,.0333333){30}{\line(-1,0){.0666667}}
\multiput(11.25,6)(.03333333,-.06666667){60}{\line(0,-1){.06666667}}
\put(13.25,2){\line(-4,3){4}}
\put(9.25,5){\line(1,0){4}}
\put(13.25,5){\line(-4,-3){4}}
\multiput(9.25,2)(.03333333,.06666667){60}{\line(0,1){.06666667}}
\multiput(13.25,5)(.0666667,.0333333){30}{\line(1,0){.0666667}}
\multiput(13.25,5)(.03333333,-.03333333){60}{\line(0,-1){.03333333}}
\multiput(20.25,6)(-.0666667,-.0333333){30}{\line(-1,0){.0666667}}
\put(18.25,5){\line(0,-1){2}}
\multiput(18.25,3)(.0666667,-.0333333){30}{\line(1,0){.0666667}}
\put(20.25,2){\line(1,0){2}}
\multiput(22.25,2)(.0666667,.0333333){30}{\line(1,0){.0666667}}
\put(24.25,3){\line(0,1){2}}
\multiput(24.25,5)(-.0666667,.0333333){30}{\line(-1,0){.0666667}}
\put(22.25,6){\line(-1,0){2}}
\multiput(20.25,6)(-.03333333,-.05){60}{\line(0,-1){.05}}
\put(18.25,3){\line(4,3){4}}
\multiput(22.25,6)(.03333333,-.05){60}{\line(0,-1){.05}}
\multiput(24.25,3)(-.1333333,-.0333333){30}{\line(-1,0){.1333333}}
\multiput(20.25,2)(-.03333333,.05){60}{\line(0,1){.05}}
\multiput(18.25,5)(.1333333,.0333333){30}{\line(1,0){.1333333}}
\put(22.25,6){\line(0,-1){4}}
\multiput(22.25,2)(-.1333333,.0333333){30}{\line(-1,0){.1333333}}
\put(18.25,3){\line(3,1){6}}
\put(24.25,5){\line(-4,-3){4}}
\put(20.25,2){\line(0,1){4}}
\put(20.25,6){\line(4,-3){4}}
\put(24.25,3){\line(-1,0){6}}
\multiput(22.25,6)(-.03333333,-.06666667){60}{\line(0,-1){.06666667}}
\multiput(20.25,6)(.03333333,-.06666667){60}{\line(0,-1){.06666667}}
\put(18.25,5){\line(3,-1){6}}
\put(18.25,5){\line(1,0){6}}
\multiput(24.25,5)(-.1333333,.0333333){30}{\line(-1,0){.1333333}}
\multiput(24.25,5)(-.03333333,-.05){60}{\line(0,-1){.05}}
\put(22.25,2){\line(-4,3){4}}
\put(22.25,6){\line(1,0){3}}
\multiput(25.25,6)(-.0333333,-.0333333){30}{\line(0,-1){.0333333}}
\multiput(25.25,6)(-.0333333,-.1){30}{\line(0,-1){.1}}
\put(27.25,6){\line(0,-1){4}}
\put(27.25,2){\line(1,0){3}}
\put(30.25,2){\line(0,1){4}}
\put(30.25,6){\line(-1,0){3}}
\put(27.25,6){\line(3,-4){3}}
\put(30.25,6){\line(-3,-4){3}}
\multiput(30.25,6)(.0333333,-.1333333){30}{\line(0,-1){.1333333}}
\multiput(30.25,6)(.03333333,-.06666667){60}{\line(0,-1){.06666667}}
\put(30.25,6){\line(3,-4){3}}
\put(.25,21){\circle*{.5}}
\put(6.25,21){\circle*{.5}}
\put(.25,18){\circle*{.5}}
\put(6.25,18){\circle*{.5}}
\put(9.25,21){\circle*{.5}}
\put(12.25,22){\circle*{.5}}
\put(15.25,21){\circle*{.5}}
\put(9.25,18){\circle*{.5}}
\put(15.25,18){\circle*{.5}}
\put(18.25,22){\circle*{.5}}
\put(20.25,22){\circle*{.5}}
\put(22.25,22){\circle*{.5}}
\put(24.25,22){\circle*{.5}}
\put(27.25,22){\circle*{.5}}
\put(29.25,22){\circle*{.5}}
\put(31.25,22){\circle*{.5}}
\put(33.25,22){\circle*{.5}}
\put(30.25,18){\circle*{.5}}
\put(21.25,18){\circle*{.5}}
\put(.25,14){\circle*{.5}}
\put(.25,10){\circle*{.5}}
\put(3.25,14){\circle*{.5}}
\put(6.25,14){\circle*{.5}}
\put(6.25,10){\circle*{.5}}
\put(3.25,10){\circle*{.5}}
\put(9.25,14){\circle*{.5}}
\put(11.25,14){\circle*{.5}}
\put(13.25,14){\circle*{.5}}
\put(15.25,14){\circle*{.5}}
\put(11.25,10){\circle*{.5}}
\put(18.25,14){\circle*{.5}}
\put(21.25,14){\circle*{.5}}
\put(24.25,14){\circle*{.5}}
\put(24.25,10){\circle*{.5}}
\put(18.25,10){\circle*{.5}}
\put(27.25,14){\circle*{.5}}
\put(29.25,13){\circle*{.5}}
\put(30.25,14){\circle*{.5}}
\put(31.25,14){\circle*{.5}}
\put(33.25,14){\circle*{.5}}
\put(29.25,10){\circle*{.5}}
\put(30.25,6){\circle*{.5}}
\put(27.25,6){\circle*{.5}}
\put(27.25,2){\circle*{.5}}
\put(30.25,2){\circle*{.5}}
\put(31.25,2){\circle*{.5}}
\put(32.25,2){\circle*{.5}}
\put(33.25,2){\circle*{.5}}
\put(25.25,6){\circle*{.5}}
\put(22.25,6){\circle*{.5}}
\put(20.25,6){\circle*{.5}}
\put(18.25,5){\circle*{.5}}
\put(18.25,3){\circle*{.5}}
\put(20.25,2){\circle*{.5}}
\put(22.25,2){\circle*{.5}}
\put(24.25,3){\circle*{.5}}
\put(24.25,5){\circle*{.5}}
\put(15.25,6){\circle*{.5}}
\put(15.25,3){\circle*{.5}}
\put(11.25,6){\circle*{.5}}
\put(9.25,5){\circle*{.5}}
\put(9.25,2){\circle*{.5}}
\put(13.25,2){\circle*{.5}}
\put(3.25,6){\circle*{.5}}
\put(6.25,6){\circle*{.5}}
\put(6.25,5){\circle*{.5}}
\put(6.25,3){\circle*{.5}}
\put(4.25,2){\circle*{.5}}
\put(2.25,2){\circle*{.5}}
\put(.25,3){\circle*{.5}}
\put(.25,5){\circle*{.5}}
\put(2.25,17){\makebox(0,0)[cc]{\scriptsize$P_4$}}
\put(2.25,16){\makebox(0,0)[cc]{\scriptsize$\partial_4\!=\!-3.14$}}
\put(12.25,17){\makebox(0,0)[cc]{\scriptsize$C_5$}}
\put(21.25,17){\makebox(0,0)[cc]{\scriptsize$H_0$}}
\put(12.25,16){\makebox(0,0)[cc]{\scriptsize$\partial_3\!=\!-0.38$}}
\put(21.25,16){\makebox(0,0)[cc]{\scriptsize$\partial_3\!=\!-0.91$}}
\put(30.25,17){\makebox(0,0)[cc]{\scriptsize$H_1$}}
\put(30.25,16){\makebox(0,0)[cc]{\scriptsize$\partial_3\!=\!-0.72$}}
\put(3.25,9){\makebox(0,0)[cc]{\scriptsize$H_2\!=\!P_3\vee P_3$}}
\put(3.25,8){\makebox(0,0)[cc]{\scriptsize$\partial_3\!=\!-0.7$}}
\put(11.25,8){\makebox(0,0)[cc]{\scriptsize$\partial_3\!=\!-0.77$}}
\put(21.25,9){\makebox(0,0)[cc]{\scriptsize$H_4\!=\!P_3\vee 2K_1$}}
\put(21.25,8){\makebox(0,0)[cc]{\scriptsize$\partial_3\!=\!-0.83$}}
\put(29.25,9){\makebox(0,0)[cc]{\scriptsize$H_5\!=\!K_1\vee(K_3\cup K_2)$}}
\put(11.25,9){\makebox(0,0)[cc]{\scriptsize$H_3\!=\!K_1\vee(P_3\cup K_1)$}}
\put(29.25,8){\makebox(0,0)[cc]{\scriptsize$\partial_6\!=\!-3.43$}}
\put(3.25,1){\makebox(0,0)[cc]{\scriptsize$I_1\!=\!K_1\vee(K_6\cup K_1)$}}
\put(3.25,0){\makebox(0,0)[cc]{\scriptsize$\partial_8\!=\!-3.07$}}
\put(12.25,1){\makebox(0,0)[cc]{\scriptsize$I_2\!=\!K_1\vee(K_4\cup2K_1)$}}
\put(12.25,0){\makebox(0,0)[cc]{\scriptsize$\partial_7\!=\!-3.21$}}
\put(21.25,1){\makebox(0,0)[cc]{\scriptsize$I_3\!=\!K_3\vee(K_5\cup K_1)$}}
\put(21.25,0){\makebox(0,0)[cc]{\scriptsize$\partial_9\!=\!-3.03$}}
\put(29.25,1){\makebox(0,0)[cc]{\scriptsize$I_4\!=\!K_1\vee(K_3\cup3K_1)$}}
\put(29.25,0){\makebox(0,0)[cc]{\scriptsize$\partial_7\!=\!-3.1$}}
\put(18.25,22.5){\makebox(0,0)[cc]{\scriptsize$v_1$}}
\put(20.25,22.5){\makebox(0,0)[cc]{\scriptsize$v_2$}}
\put(22.25,22.5){\makebox(0,0)[cc]{\scriptsize$v_3$}}
\put(24.25,22.5){\makebox(0,0)[cc]{\scriptsize$v_4$}}
\put(27.25,22.5){\makebox(0,0)[cc]{\scriptsize$v_1$}}
\put(29.25,22.5){\makebox(0,0)[cc]{\scriptsize$v_2$}}
\put(31.25,22.5){\makebox(0,0)[cc]{\scriptsize$v_3$}}
\put(33.25,22.5){\makebox(0,0)[cc]{\scriptsize$v_4$}}
\put(21.75,18){\makebox(0,0)[cc]{\scriptsize$v$}}
\put(30.75,18){\makebox(0,0)[cc]{\scriptsize$v$}}
\put(13.25,5){\circle*{.5}}
\end{picture}
\caption{Forbidden induced subgraphs}\label{fig-1}
\end{figure}
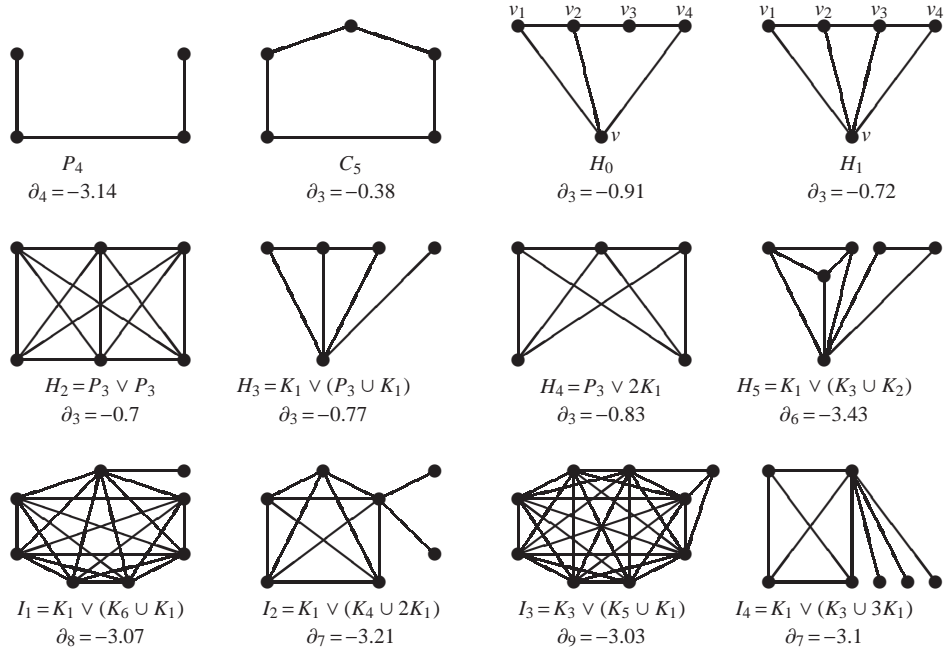

\begin{lem}\label{lem-3-11}
If $G\in\mathcal{G}[-3\le\partial_n,\partial_3\le-1]$, then there exists two proper subgraphs $G_1$ and $G_2$ at most one of them containing induced $P_3$ such that $G=G_1\vee G_2$. Furthermore, if $G_1$ contains induced $P_3$, then $G_1$ is connected and $G_2$ is complete.
\end{lem}
\begin{proof}
By Lemma \ref{lem-3-10},  $G$ contains no induced $P_4$ and so $G=G_1\vee G_2$ by Lemma \ref{lem-3-9}. By Lemma \ref{lem-3-10} again, $G$ contains no induced $H_2$ $(\cong P_3\vee P_3)$ and so at most one of $G_1$ and $G_2$ contains induced $P_3$. Furthermore, if $G_1$ contains induced $P_3$ but $G_2$ does not, then $G_1$ is connected since otherwise $G$ will contain induced $H_3$ $(\cong(P_3\cup K_1)\vee K_1)$, and $G_2$ is a union of complete graphs since it contains no induced $P_3$. In fact, $G_2\cong K_s$ for some $s\ge 1$, since otherwise $G$ will contain induced $H_4$ $(\cong P_3\vee 2K_1)$, which contradicts Lemma \ref{lem-3-10}.
\end{proof}

 Lemma \ref{lem-3-11} gives a sketch for the graph $G$ in $\mathcal{G}[-3\le\partial_n,\partial_3\le-1]$, that is, $G=G_1\vee G_2$. Now we  give a precise characterization of $G_1$ and $G_2$ in the following theorem.
\begin{thm}\label{thm-3-12}Let $G=G_1\vee G_2\in \mathcal{G}[-3\le\partial_n,\partial_3\le-1]$.\\
(i) If both $G_1$ and $G_2$ contain no induced $P_3$, then $G$ is one of the followings:

(a) $G\cong K_r\vee(K_s\cup tK_1)$ where $r,t\ge1$ and $s\ge3$;

(b) $G\cong K_r\vee(mK_2\cup tK_1)$ where $r\ge1$ and $m+t\ge2$;

(c) $G\cong (K_{s_1}\cup t_1K_1)\vee (K_{s_2}\cup t_2K_1)$ where $s_i\ge 3$ and $t_i\ge1$;

(d) $G\cong (K_s\cup t_1K_1)\vee(mK_2\cup t_2K_1)$ where $s\ge3,t_1\ge1$ and $m+t_2\ge2$;

(e) $G\cong (m_1K_2\cup t_1K_1)\vee(m_2K_2\cup t_2K_1)$ where $m_i+t_i\ge2$ for $i=1,2$.\\
(ii) If one of $G_1$ and $G_2$ contains induced $P_3$, then there exists $K_r$ such that $G=K_r\vee F_1\vee F_2$, where $r\ge1$, $F_1$ and $F_2$ are non-null containing no induced $P_3$.
\end{thm}
\begin{proof}
By Lemma \ref{lem-3-11}, there exists two proper subgraphs $G_1$ and $G_2$ at most one of them containing induced $P_3$ such that $G=G_1\vee G_2$. Now we divide our proof into two cases.

\textbf{Case 1.} Both $G_1$ and $G_2$ contain no induced $P_3$.

In this case, both $G_1$ and $G_2$ are unions of some complete graphs. Then at most one of $G_1$ and $G_2$ is connected since otherwise $G$ will be complete. We consider two subcases bellow.

\textbf{Subcase 2.1.} One of $G_1$ and $G_2$ is connected;

Without loss of generality, we assume that $G_1$ is connected but $G_2$ is disconnected. Then $G_1\cong K_r$ for some $r\ge 1$.

If $G_2$ contains $K_3$, then  $G_2\cong K_s\cup tK_1$, where $s\ge 3,t\ge1$. Since otherwise, $G_2$ will contain induced $K_3\cup K_2$, and then $G=K_r\vee G_2$ will contain induced $H_5$ $(\cong K_1\vee(K_3\cup K_2))$,  which contradicts Lemma \ref{lem-3-10}. Thus (a) follows.

If $G_2$ contains no $K_3$, then $G_2\cong mK_2\cup tK_1$, where $m+t\ge2$. It follows (b).

\textbf{Subcase 2.2.} Both of $G_1$ and $G_2$ are disconnected.

If both of $G_1$ and $G_2$ contain $K_3$, then $G_1\cong K_{s_1}\cup t_1K_1$ and $G_2\cong K_{s_2}\cup t_2K_1$, where $s_i\ge3$ and $t_i\ge 1$ for $i=1,2$. Since otherwise, $G_1$ or $G_2$ will contain induced $K_3\cup K_2$, and then $G$ will contain induced $H_5$, which contradicts Lemma \ref{lem-3-10}. It follows (c).

If just one of $G_1$ and $G_2$ contains $K_3$, say $G_1$, then $G_2\cong mK_2\cup t_2K_1$, where  $m+t_2\ge2$. We claim that  $G_1\cong K_s\cup t_1K_1$ for $s\ge3,t\ge1$. Since otherwise, $G_1$  will contain induced $K_3\cup K_2$, and thus $G$ will contain induced $H_5$, which contradicts Lemma \ref{lem-3-10}. It follows (d).

If both  $G_1$ and $G_2$  contain no $K_3$, then $G_1\cong m_1K_2\cup t_1K_1$ and $G_2\cong m_2K_2\cup t_2K_1$, where $m_i+t_i\ge2$ for $i=1,2$. It follows  (e).

 \textbf{Case 2.} Exactly one of $G_1$ and $G_2$ contains induced $P_3$.

 Without loss of generality, suppose $G_1$ contains induced $P_3$ but $G_2$ does not. By Lemma \ref{lem-3-11}, $G_1$ is connected and $G_2\cong K_s$ for some $s\ge 1$. By Lemma \ref{lem-3-10},  $G$ contains no induced $P_4$, so $G_1$ contains no induced $P_4$. By Lemma \ref{lem-3-9}, $G_1$ is the join of two non-null graphs, and so the diameter of $G_1$ is less than $3$. Now we obtain that $G=K_s\vee G_1$, where $d(G_1)<3$. Thus the eigenvalues of $D(G_1)$ interlace that of $D(G)$ by Corollary \ref{cor-3-3}, and so $G_1\in \mathcal{G}[-3\le\partial_{|G_1|},\partial_3\le-1]$. Again by Lemma \ref{lem-3-11}, we have $G_1=G_1'\vee G_2'$, in which at most one of $G_1'$ and $G_2'$ contains induced $P_3$. Thus $G=K_s\vee G_1'\vee G_2'$, where $s\ge1$.

Now, we may assume that $K_r$ $(r\ge1)$ is the maximum clique such that  $G=K_r\vee F_1\vee F_2$, where $F_1$ and $F_2$ are  non-null and at most one of  them contains induced $P_3$. Finally, we show that $F_1$ and $F_2$ contain no induced $P_3$. By the way of contradiction, suppose $F_1$ contains  induced $P_3$ but  $F_2$ does not. Let $F=F_1\vee F_2$. Since the diameter of $F$ is less than $3$, we have $\partial_3(F)\le\partial_3(G)\le-1$ and $\partial_{|F|}(F)\ge\partial_n(G)\ge-3$ by Corollary \ref{cor-3-3}. By Lemma \ref{lem-3-11}, $F_1$ is connected and $F_2$ is complete, say $F_2=K_{r'}$ $(r'\ge1)$.  By Lemma \ref{lem-3-10}, $G$ contains no induced $P_4$, nor does $F_1$. Thus there exists non-null graphs $F_1',F_2'$ such that $F_1=F_1'\vee F_2'$ by Lemma \ref{lem-3-9}, and so $G\cong K_{r+r'}\vee$ $(F_1'\vee F_2')$, which is a contradiction since $r+r'>r$.
\end{proof}
Theorem \ref{thm-3-12} tells us that those graphs belonging to $\mathcal{G}[-3\le\partial_n,\partial_3\le-1]$ are included in the set of graphs described in Theorem \ref{thm-3-12}. Conversely, the graphs described in Theorem \ref{thm-3-12} may not be in $\mathcal{G}[-3\le\partial_n,\partial_3\le-1]$. In fact, the graph $G=K_1\vee(K_6\cup K_1)$ has the form $K_r\vee(K_s\cup K_1)$ characterized in Theorem \ref{thm-3-12}(i)(a), however $G\not\in\mathcal{G}[-3\le\partial_n,\partial_3\le-1]$ since $\mathrm{Spec}_D(G)=[8.78,-0.70,(-1)^5,-3.07]$. Naturally, we try to give a complete characterization of the graphs in $\mathcal{G}[-3\le\partial_n,\partial_3\le-1]$. We first classify those graphs described in Theorem \ref{thm-3-12} into three types.

By simple observation of Theorem \ref{thm-3-12}, all  graphs characterized in Theorem \ref{thm-3-12}(i) can be written as $K_0\vee F_1\vee F_2$, where $F_1$ and $F_2$ are non-null and contain no induced $P_3$. Therefore, we get that $G\cong K_r\vee  F_1\vee F_2$ if $G\in \mathcal{G}[-3\le\partial_n,\partial_3\le-1]$, where $r\ge0$, $F_1$ and $F_2$ are non-null and contain no induced $P_3$.
A graph $F$ containing no  induced $P_3$ will be a complete graph if it is connected, and will be a union of some complete graphs otherwise. Therefore, we get the following result.
\begin{cor}\label{cor-3-13}
If $G\in \mathcal{G}[-3\le\partial_n,\partial_3\le-1]$, then one of the following cases holds:\\
(I) $G\cong K_r\vee (K_{s_1}\cup\cdots\cup K_{s_i})$, where $r\ge1$ and $i\ge 2$; \\
(II) $G\cong (K_{s_1}\cup\cdots\cup K_{s_i})\vee  (K_{t_1}\cup\cdots\cup K_{t_j})$, where $i\ge 2$ and $j\ge 2$.
\end{cor}
\begin{proof}
First we know that $G=K_r\vee F_1\vee F_2$, where $F_i$ $(i=1,2)$ is the union of some complete graphs and $r\ge0$. Since $G$ is not complete, at most one of $F_1$ and $F_2$ is connected. If exactly one of $F_1$ and $F_2$ is connected,  then $G\cong K_r\vee (K_{s_1}\cup\cdots\cup K_{s_i})$, where $r\ge1$ and $i\ge 2$. Thus (I) holds. If both of $F_1$ and $F_2$ are disconnected, then $G\cong K_r\vee (K_{s_1}\cup\cdots\cup K_{s_i})\vee (K_{t_1}\cup\cdots\cup K_{t_j})$,  where $r\ge0$, $i\ge 2$ and $j\ge 2$. By Lemma \ref{lem-3-10}, $H_4$ $(=K_1\vee2K_1\vee2K_1)$ cannot be an induced subgraph of $G$, we claim that $r=0$ and so (II) holds.
\end{proof}

We say that $G\in\mathcal{G}[-3\le\partial_n,\partial_3\le-1]$ is of type-I and type-II if $G$ satisfies (I) and (II) in Corollary \ref{cor-3-13}, respectively.

Next we give a complete characterization of the graphs in $\mathcal{G}[-3\le\partial_n,\partial_3\le-1]$ by using the forbidden subgraphs $H_4,H_5,I_1,I_2,I_3,I_4$. Denote by $S(m,n)=(mK_2\cup nK_1)\vee(K_5\cup K_1)$ ($m+n\ge1$), $T_1=K_4\cup K_1$, $T_2=K_3\cup 2K_1$, $T_3=K_3\cup K_1$ and $T_4(m,n)=mK_2\cup nK_1$ ($m+n\ge2$). Moreover, denote by $\mathcal{S}=\{S(m,n)\mid m+n\ge1\}$, $\mathcal{T}_1=\{K_r\vee T_i\mid r\ge1, 1\le i\le3\}\cup\{K_r\vee T_4(m,n)\}$ and $\mathcal{T}_2=\{T_i\vee T_j\mid 1\le i,j\le4\}\cup\{T_i\vee T_4(m,n), T_4(m_1,n_1)\vee T_4(m_2,n_2)\mid 1\le i\le 3\}$. It is the stage to state the one of our main  result ( Theorem \ref{thm-3-14} ) whose proof will be given latter after Lemma \ref{lem-3-15} and Lemma \ref{lem-3-16}.
\begin{thm}\label{thm-3-14}
$\mathcal{G}[-3\le\partial_n,\partial_3\le-1]=\mathcal{S}\cup\mathcal{T}_1\cup\mathcal{T}_2.$
\end{thm}
For the convenience, we first partition the graphs in $\mathcal{S}\cup\mathcal{T}_1\cup\mathcal{T}_2$ into $32$ subclasses  in terms of parameters $m$ and $n$ in the following table:\\
\renewcommand{\arraystretch}{0.95}
\begin{table}[H]
\vspace{-0.4cm}
\floatsetup{floatrowsep=qquad,captionskip=5pt} \tabcolsep=9pt
\begin{center}{\scriptsize
\begin{floatrow}
\ttabbox{\caption{\label{tab-1}\scriptsize{Subclasses of $\mathcal{S}\cup\mathcal{T}_1\cup\mathcal{T}_2$}}}{
\begin{tabular}{lll}
  \toprule
     &type-I &type-II \\
  \midrule
  $\mathcal{S}$&$S(1,0)$, $S(0,1)$&$S(m,0)$ ($m\ge2$), $S(0,n)$ ($n\ge2$), $S(m,n)$ ($m,n\ge1$)\\
  $\mathcal{T}_1$&\tabincell{l}{$K_r\vee T_1$, $K_r\vee T_2$, $K_r\vee T_3$,\\ $K_r\vee T_4(m,0)$, $K_r\vee T_4(0,n)$, $K_r\vee T_4(m,n)$}& \\
  $\mathcal{T}_2$& &\tabincell{l}{$T_1\vee T_1$,$T_1\vee T_2$, $T_1\vee T_3$,$T_1\vee T_4(m,n)$,$T_1\vee T_4(m,0)$,$T_1\vee T_4(0,n)$,\\ $T_2\vee T_2$,$T_2\vee T_3$,$T_2\vee T_4(m,n)$,$T_2\vee T_4(m,0)$,$T_2\vee T_4(0,n)$,$T_3\vee T_3$,\\$T_3\vee T_4(m,n)$,$T_3\vee T_4(m,0)$,$T_3\vee T_4(0,n)$,$T_4(m_1,n_1)\vee T_4(m_2,n_2)$,\\$T_4(m_1,n_1)\vee T_4(m_2,0)$,$T_4(m_1,n_1)\vee T_4(0,n_2)$,$T_4(m_1,0)\vee T_4(m_2,0)$,\\$T_4(m_1,0)\vee T_4(0,n_2)$,$T_4(0,n_1)\vee T_4(0,n_2)$}\\
  \bottomrule
 \end{tabular}}
\end{floatrow}}
\end{center}
\vspace{-0.4cm}
\end{table}

We calculate all the distance spectra of the graphs in Table \ref{tab-1}, which are list in Appendix A. In fact, we concretely calculate some the distance spectra of them in details in Lemma \ref{lem-3-16} and the others will be obtained by the same method.

\begin{lem}\label{lem-3-15}
$\mathcal{G}[-3\le\partial_n,\partial_3\le-1]\subseteq\mathcal{S}\cup\mathcal{T}_1\cup\mathcal{T}_2$.
\end{lem}
\begin{proof}
Let $G\in \mathcal{G}[-3\le\partial_n,\partial_3\le-1]$. Then $G$ is of type-I or type-II by Corollary \ref{cor-3-13}.

Suppose that $G$ is of type-I, i.e., $G\cong K_r\vee (K_{s_1}\cup\cdots\cup K_{s_i})$, where $r\ge1$ and $i\ge 2$. Without loss of generality, we assume that $s_1\ge\cdots\ge s_i$. Since $I_1$ $(=K_1\vee(K_6\cup K_1))$ cannot be an induced subgraph of $G$, $s_1\le5$. If $s_1=5$, then $s_k\le 1$ for $k=2,...,i$ because $G$ contains no induced $H_5$ $(=K_1\vee(K_3\cup K_2))$. Moreover, since $I_2$ $(=K_1\vee(K_4\cup 2K_1))$ cannot be an induced subgraph of $G$, we have $s_3=0$, which implies that $G=K_r\vee(K_5\cup K_1)$. Note that $I_3$ $(=K_3\vee(K_5\cup K_1))$ cannot be an induced subgraph of $G$, we obtain that $r=1$ or $r=2$, that is, $G=K_1\vee(K_5\cup K_1)=S(0,1)$ or $G=K_2\vee(K_5\cup K_1)=S(1,0)$. If $s_1=4$, then $s_k\le 1$ for $k=2,...,i$ because $G$ contains no induced $H_5$. Moreover, since $I_2$ cannot be an induced subgraph of $G$, we have $s_3=0$, which implies that $G=K_r\vee(K_4\cup K_1)=K_r\vee T_1$. Similarly, if $s_1=3$, then $G=K_r\vee (K_3\cup 2K_1)=K_r\vee T_2$ or $G=K_r\vee (K_3\cup K_1)=K_r\vee T_3$ because $H_5$ and $I_4$ ($=K_1\vee(K_3\cup3K_1)$) cannot be induced subgraphs of $G$. If $s_1\le2$ then $G=K_r\vee (mK_2\cup nK_1)=K_r\vee T_4(m,n)$.

Suppose that $G$ is of type-II, i.e., $G\cong (K_{s_1}\cup\cdots\cup K_{s_i})\vee  (K_{t_1}\cup\cdots\cup K_{t_j})$, where $i\ge 2$ and $j\ge 2$. Without loss of generality, we assume that $s_1\ge\cdots\ge s_i$,  $t_1\ge\cdots\ge t_j$ and $s_1\ge t_1$.  Since $I_1$ $(=K_1\vee(K_6\cup K_1))$ cannot be an induced subgraph of $G$, $s_1\le5$. If $s_1=5$, then $s_2=1$ and $s_3=0$ because $H_5$ and $I_2$ cannot be induced subgraphs of $G$. Moreover, we have $t_1\le 2$ because $I_3$ cannot be an induced subgraph of $G$. Thus $G=(K_5\cup K_1)\vee(mK_2\cup nK_1)=S(m,n)$ ($m+n\ge2$). Similarly, If $s_1=4$, then $G=T_1\vee T_i$ for $1\le i\le 3$ or $G=T_1\vee T_4(m,n)$ because $H_5$, $I_2$ and $I_4$ cannot be induced subgraphs of $G$; if $s_1=3$, then $G=T_i\vee T_j$ or $G=T_i\vee T_4(m,n)$ for for $2\le i,j\le 3$ because $H_5$ and $I_4$ cannot be induced subgraphs of $G$; if $s_1\le2$, then $G=T_4(m_1,n_1)\vee T_4(m_2,n_2)$.

We complete the proof.
\end{proof}

\begin{lem}\label{lem-3-16}
$\mathcal{S}\cup\mathcal{T}_1\cup\mathcal{T}_2\subseteq \mathcal{G}[-3\le\partial_n,\partial_3\le-1]$.
\end{lem}
\begin{proof}
By calculating the distance spectra, we obtain that  $\{S(0,1),S(1,0)\}\cup\{T_i\vee T_j\mid 1\le i,j\le 3\}\subseteq\mathcal{G}[-3\le\partial_n,\partial_3\le-1]$ (see Appendix A). It suffices to  consider the remaining situations.

First we consider the graph $K_r\vee T_4(m,n)=K_r\vee(mK_2\cup nK_1)\in \mathcal{T}_1$, where $m+n\geq 2$. If $m,n\ne0$, it is easy to see that $\Pi$: $V(K_r\vee T_4(m,n))=V(K_r)\cup V(mK_2)\cup V(nK_1)$ is a distance equitable partition of $K_r\vee T_4(m,n)$ with the distance divisor matrix
$$
B_1=\left(
    \begin{array}{ccc}
      r-1 & 2m & n \\
      r & 4m-3 & 2n \\
      r & 4m & 2n-2 \\
    \end{array}
  \right).
$$
By Theorem \ref{thm-2-3}, $\det(xI-B_1)|\det(xI-D(K_r\vee T_4(m,n)))$, which implies that the roots of  the polynomial
\begin{eqnarray*}
f_1(x)&=&\det(xI-B_1)=x^3+(6-2n-4m-r)x^2\\
&+&(2mr-8n-5r-12m+nr+11)x-(8m+6n+6r-4mr-3nr-6)
\end{eqnarray*}
are  distance eigenvalues of $K_r\vee T_4(m,n)$.
Note that $f_1(-1)=2r(m+n-1)>0>f_1(-3)=-8m-2mr$. By Corollary \ref{cor-2-4}, the largest root of $f_1(x)$ is just the distance spectral radius of $K_r\vee T_4$, which is simple and greater than $0$. Therefore, by the function image of $f_1(x)$, the roots of $f_1(x)$ satisfy $-3<\partial_3<-1<\partial_2<\partial_1$. Moreover, by Lemma \ref{lem-3-4}, Corollary \ref{cor-3-6} and Corollary \ref{cor-3-7}, $-1$, $-2$ and $-3$ are distance eigenvalues of $G$ with multiplicities at least $r-1$, $n-1$ and $m-1$, respectively. Thus  $\mathrm{Spec}_D(K_r\vee T_4(m,n))=[\partial_1,\partial_2,(-1)^{m+r-1},(-2)^{n-1},\partial_3,(-3)^{m-1}]$ and so $K_r\vee T_4(m,n)\in \mathcal{G}[-3\le\partial_n,\partial_3\le-1]$. If $m=0$ or $n=0$, by using the same  method, we get the distance spectrum of $K_r\vee T_4(m,0)$ and $K_r\vee T_4(0,n)$ (see Appendix A), and thus we have $K_r\vee T_4(m,0),K_r\vee T_4(0,n)\in \mathcal{G}[-3\le\partial_n,\partial_3\le-1]$ as well.
Similarly, we get the distance spectra of $K_r\vee T_i$ for $i=1,2,3$ (see Appendix A). Clearly, all these graphs belong to $\mathcal{G}[-3\le\partial_n,\partial_3\le-1]$.

Next we consider the graph $S(m,n)=(K_5\cup K_1)\vee(mK_2\cup nK_1)\in \mathcal{S}$, where $m+n\ge2$. If $m,n\ne0$, it is easy to see that $\Pi$: $V(S(m,n))=V(K_5)\cup V(K_1)\cup V(mK_2)\cup V(nK_1)$ is a distance equitable partition of $S(m,n)$ with the distance divisor matrix
$$
B_2=\left(
    \begin{array}{cccc}
      4 & 2 & 2m & n \\
      10 & 0 & 2m & n \\
      5 & 1 & 4m-3 & 2n \\
      5 & 1 & 4m & 2n-2 \\
    \end{array}
  \right).
$$
By Theorem \ref{thm-2-3}, $\det(xI-B_2)|\det(xI-D(S(m,n)))$. It follows that the roots of the polynomial
$$f_2(x)=(x+3)(x^3-(2n+4m+2)x^2+(2n+8m-28)x+32m+24n-40)$$
are distance eigenvalues of $S(m,n)$.
One root of $f_2(x)$ is $-3$, and the others are the roots of
$g(x)=x^3-(2n+4m+2)x^2+(2n+8m-28)x+32m+24n-40$.
Note that $g(-1)=20(m+n)-15>0>g(-3)=-28n-1$ and the largest root of $g(x)$ is greater than $0$. Thus, by the function image of $g(x)$, the roots of $g(x)$ satisfy $-3<\partial_3<-1<\partial_2<\partial_1$. Moreover, by Lemma \ref{lem-3-4}, Corollary \ref{cor-3-6} and Corollary \ref{cor-3-7}, $-1$, $-2$ and $-3$ are distance eigenvalues of $G$ with multiplicity at least $m+4$, $n-1$ and $m-1$, respectively. Thus $\mathrm{Spec}_D(S(m,n))=[\partial_1,\partial_2,(-1)^{m+4},(-2)^{n-1},\partial_3,(-3)^m]$, and so $S(m,n)\in \mathcal{G}[-3\le\partial_n,\partial_3\le-1]$. If $m=0$ or $n=0$, by using the same method, we also get the distance spectrum of $S(m,n)$ (see Appendix A), and so $S(m,n)\in\mathcal{G}[-3\le\partial_n,\partial_3\le-1]$.

At last we consider the graph $T_1\vee T_4(m,n)=(K_4\cup K_1)\vee(mK_2\cup nK_1)\in \mathcal{T}_2$, where $m+n\ge2$. If $m,n\ne0$, it is easy to see that $\Pi$: $V(T_1\vee T_4(m,n))=V(K_4)\cup V(K_1)\cup V(mK_2)\cup V(nK_1)$ is a distance equitable partition of $T_1\vee T_4(m,n)$ with the distance divisor matrix
$$
B_3=\left(
    \begin{array}{cccc}
      3 & 2 & 2m & n \\
      8 & 0 & 2m & n \\
      4 & 1 & 4m-3 & 2n \\
      4 & 1 & 4m & 2n-2 \\
    \end{array}
  \right).
$$
By Theorem \ref{thm-2-3}, $\det(xI-B_3)|\det(xI-D(T_1\vee T_4(m,n)))$. Thus the roots of the polynomial
$$f_3(x)\!=\!x^4\!+\!(2-2n-4m)x^3\!-\!(6m+5n+25)x^2\!+\!(42m+22n-98)x\!+\!76m\!+\!57n\!-\!96.$$
are  distance eigenvalues of $T_1\vee T_4$.
Note that the derivative of $f_3(x)$ is
$$f'_3(x)=4x^3-(12m+6n-6)x^2-2(6m+5n+25)x+42m+22n-98.$$
By simple computation, we have $f_3(-1)=32m+32n-24>0$, $f_3(-3)=4m>0$, $f'_3(-3)=-30m-2n-2<0$, $f'_3(-1)=42m+26n-46>0$ and $f'_3(3)=-102m-62n-86<0$. By the function image of $f'_3(x)$, the roots of $f'_3(x)$, denoted by $\mu_1,\mu_2,\mu_3$, satisfy $-3<\mu_3<-1<\mu_2<3<\mu_1$. Therefore $f_3(x)$ monotonically decreases when $x<-3$. Since $f_3(-3)>0$, we have $f_3(x)>0$ for $x\le-3$. Moreover, since $-3<\mu_3<-1$ and $f_3(-1),f_3(-3)>0$, by the function image of $f_3(x)$, we obtain that two roots of $f_3(x)$ lie in the interval $(-3,-1)$, and the other two roots lie in $(-1,+\infty)$. Combining Lemma \ref{lem-3-4}, Corollary \ref{cor-3-6} and Corollary \ref{cor-3-7}, we obtain the distance spectrum of $T_1\vee T_4(m,n)$ (see Appendix A), and so $T_1\vee T_4(m,n)\in\mathcal{G}[-3\le\partial_n,\partial_3\le-1]$. If $m=0$ or $n=0$, by using the same method, we also get the distance spectrum of $T_1\vee T_4(m,0)$ and $T_1\vee T_4(0,n)$ (see Appendix A) and $T_1\vee T_4(m,0),T_1\vee T_4(0,n)\in \mathcal{G}[-3\le\partial_n,\partial_3\le-1]$ as well.
Similarly, we get the distance spectra of $T_i\vee T_4(m,n)$ for $i=2,3,4$ (see Appendix A), and all these graphs belong to $\mathcal{G}[-3\le\partial_n,\partial_3\le-1]$.

We complete the proof.
\end{proof}
\begin{remark}
\emph{In the proof of Lemma \ref{lem-3-16}, we compute the distance spectra of three subclasses in Table \ref{tab-1} in detail. In general, for a graph $G$ in Table \ref{tab-1}, by Theorem \ref{thm-2-3}, we first get some distance eigenvalues by analysing the roots of the polynomial of the corresponding distance divisor matrix. Then we get the other distance eigenvalues, which are $-1$, $-2$ or $-3$, by using Lemma \ref{lem-3-4}, Corollary \ref{cor-3-6} and Corollary \ref{cor-3-7}. At last, by noticing that the number of distance eigenvalues equals to the order of $G$, we get the distance spectrum of $G$. Repeating these process, we obtain the spectra of all graphs in Table \ref{tab-1} and we list them in Appendix A.}
\end{remark}
\begin{proof}[ Proof of Theorem\ref{thm-3-14}]
Lemma \ref{lem-3-15} shows that $\mathcal{G}[-3\le\partial_n,\partial_3\le-1]\subseteq \mathcal{S}\cup\mathcal{T}_1\cup\mathcal{T}_2$ and Lemma \ref{lem-3-16} shows that $\mathcal{S}\cup\mathcal{T}_1\cup\mathcal{T}_2\subseteq \mathcal{G}[-3\le\partial_n,\partial_3\le-1]$. Combining them, we get Theorem \ref{thm-3-14}.
\end{proof}

\section{The graphs with exactly two distance eigenvalues different from $-1,-3$}\label{s-4}

Let $\mathcal{F}$ denote the set of connected graphs with all but two adjacency eigenvalues equal to $\pm1$. Then $\mathcal{F}$ is a special family of graphs with exactly four distinct adjacency eigenvalues. Cioab\u{a} et al.\cite{Sebastian} completely characterized the graphs in $\mathcal{F}$. Motivated by their work, we try to characterize a special family of graphs with exactly four distinct distance eigenvalues, that is, the  graphs having exactly two distance eigenvalues (counting multiplicity) different from $-1$ and $-3$.  Let $\mathcal{H}$ denote the set of such graphs.  In this section, we will give a complete characterization of the graphs in $\mathcal{H}$.
\begin{lem}\label{lem-4-1}
The graph set $\mathcal{H}\subseteq\mathcal{G}[-3\le\partial_n,\partial_3\le-1]$.
\end{lem}
\begin{proof}
Let $G\in\mathcal{H}$, and let $\alpha,\beta$ ($\alpha>\beta$) be the two distance eigenvalues of $G$ different from $-1,-3$. We claim that $\beta>-1$. Otherwise, we have $\partial_2(G)=-1<-0.73=\partial_2(P_3)$ ($\mathrm{Spec}_D(P_3)=[2.73,-0.73,-2]$). This implies that $G$ contains no induced $P_3$ and so is a complete graph, which is impossible because complete graphs have only two distinct distance eigenvalues. Hence, $G\in \mathcal{G}[-3\le\partial_n,\partial_3\le-1]$ and our results follows.
\end{proof}

According to Lemma \ref{lem-4-1}, we determine all graphs belonging to $\mathcal{H}$ in the following theorem.
\begin{thm}\label{thm-4-2}
A connected graph $G$ has exactly two distance eigenvalues (counting multiplicity) different from $-1$ and $-3$ (i.e., $G\in \mathcal{H}$) if and only if \\
 (i) $G\cong(K_5\cup K_1)\vee mK_2$ ($m\ge1$), or\\
 (ii) $G\cong K_r\vee mK_2$ ($r\ge1$, $m\ge2$), or\\
 (iii) $G\cong m_1K_2\vee m_2K_2$ ($m_1,m_2\ge 2$).
\end{thm}
\begin{proof}
By Appendix A, we see that $(K_5\cup K_1)\vee mK_2$, $K_r\vee mK_2$ and $m_1K_2\vee m_2K_2$ have exactly two distance eigenvalues different from $-1$ and $-3$. Thus the sufficiency follows. Conversely, let $G$ be a graph with exactly two distance eigenvalues different from $-1$ and $-3$, i.e., $G\in\mathcal{H}$. By Lemma \ref{lem-4-1}, we know that $G\in\mathcal{G}[-3\le\partial_n,\partial_3\le-1]$. Thus the necessity follows because $(K_5\cup K_1)\vee mK_2$, $K_r\vee mK_2$ and $m_1K_2\vee m_2K_2$ are the only graphs in $\mathcal{G}[-3\le\partial_n,\partial_3\le-1]$ with exactly two distance eigenvalues different from $-1$ and $-3$ by Appendix A.
\end{proof}
By Theorem \ref{thm-4-2}, we obtain the following result.
\begin{thm}\label{thm-4-3}
All graphs with exactly two distance eigenvalues (counting multiplicity) different from $-1$ and $-3$  are determined by their distance spectra.
\end{thm}
\begin{proof}
By Theorem \ref{thm-4-2},  every graph in $\mathcal{H}$ has the form  $(K_5\cup K_1)\vee mK_2$ ($m\ge1$), or $K_r\vee mK_2$ ($r\ge1$, $m\ge2$), or $m_1K_2\vee m_2K_2$ ($m_1,m_2\ge 2$).
By Appendix A, we get the distance spectra of these graphs:
\begin{equation}\label{eq-proof-1}
\left\{\begin{array}{l}
\mathrm{Spec}_D((K_5\cup K_1)\vee mK_2)\!=\![2m\!+\!2\pm2\sqrt{m^2\!-\!2m\!+\!6},(-1)^{m+4},(-3)^m]\\
\mathrm{Spec}_D(K_r\vee mK_2)\!=\![2m\!+\!\frac{r}{2}\!-\!2\pm\frac{\sqrt{(4m\!-\!2)^2\!+\!(r\!+\!2)^2\!-\!4}}{2},(-1)^{m+r-1},(-3)^{m-1}]\\
\mathrm{Spec}_D(m_1K_2\vee m_2K_2)\!=\![2m_1\!+\!2m_2\!-\!3\pm2\sqrt{m_1^2\!-\!m_1m_2\!+\!m_2^2},(-1)^{m_1+m_2},(-3)^{m_1+m_2-2}]\end{array}
\right.
\end{equation}
It is easy to verify that any two  graphs of the same form with different parameters cannot share the same distance spectrum. Thus we only need to consider the distance spectra of these graphs of distinct form.

First suppose that $\textrm{Spec}_D((K_5\cup K_1)\vee mK_2)=\mathrm{Spec}_D(K_r\vee m'K_2)$. By counting the multiplicities of $-3$ and $-1$, we have $m'-1=m$ and $m'+r-1=m+4$, which leads to $m'=m+1$ and $r=4$. Furthermore, by comparing the distance spectral radius we get
$$2m+2+2\sqrt{m^2-2m+16}=2m'+\frac{r}{2}-2+\frac{\sqrt{(4m'-2)^2+(r+2)^2-4}}{2}.$$
Putting $m'=m+1$ and $r=4$ into the above equation, we obtain  $m=\frac{5}{4}$, which is impossible.

Next suppose that $\textrm{Spec}_D((K_5\cup K_1)\vee mK_2)=\mathrm{Spec}_D(m_1K_2\vee m_2K_2)$. By counting the multiplicities of $-3$ and $-1$, we have
$m_1+m_2-2=m$  and $m_1+m_2=m+4$.
Therefore, $m+2=m+4$, a contradiction.

At last we suppose that $\mathrm{Spec}_D(K_r\vee m'K_2)=\mathrm{Spec}_D(m_1K_2\vee m_2K_2)$. By counting the multiplicities of $-3$ and $-1$, we have $m_1+m_2-2=m'-1$ and $m_1+m_2=m'+r-1$, which implies that $m_1+m_2=m'+1$ and $r=2$. Furthermore, by comparing the distance spectral radius, we have
$$2m'+\frac{r}{2}-2+\frac{\sqrt{(4m'-2)^2+(r+2)^2-4}}{2}=2m_1+2m_2-3+2\sqrt{m_1^2-m_1m_2+m_2^2}.$$
Putting $m_1+m_2=m'+1$ and $r=2$ into the above equation, we obtain that $m_1m_2=m'=m_1+m_2-1$, which is impossible because $m_1m_2\ge m_1+m_2$ due to $m_1,m_2\ge2$.

We complete the proof.
\end{proof}
\begin{remark}
\emph{From Theorem \ref{thm-4-3}, we know that  $\mathcal{H}$ has three classes of graphs: $(K_5\cup K_1)\vee mK_2\in \mathcal{S}$, $K_r\vee mK_2\in \mathcal{T}_1$ and $m_1K_2\vee m_2K_2\in \mathcal{T}_2$. They are all determined by distance spectra. In fact, we have confirmed some other graphs in Appendix A are also DDS. We guess that all the graphs in $G[-3\le\partial_n,\partial_3\le-1]$ are determined by their distance spectra.  We believe, if necessary,  one can verify the guess by comparing the distance spectra of the graphs in Appendix A as in the proof of Theorem \ref{thm-4-3}. Here we abandon the verifications and leave them to someone interested.}
\end{remark}

Notice that the friendship graph $F_k=K_1\vee kK_2$ is included in $K_r\vee mK_2$. The following result follows from Theorem \ref{thm-4-3} immediately.
\begin{cor}\label{cor-4-4}
The friendship graph $F_k=K_1\vee kK_2$ is determined by its distance spectrum.
\end{cor}
Corollary \ref{cor-4-4} provides of witness that distance spectrum is stronger than adjacency spectrum since the friendship graph $F_{16}$ is not DAS.

\begin{flushleft}
\textbf{Acknowledgments}
\end{flushleft}
The authors are grateful to the referees for their helpful comments  and suggections  that have improved our result in this paper.

\section*{Appendix A. Spectra of graphs in $\mathcal{G}[-3\le\partial_n,\partial_3\le-1]$}
\renewcommand{\arraystretch}{0.95}
\begin{sidewaystable}[H]

\centering
\scriptsize
\begin{tabular}{lll}
  \toprule
Graphs &Distance Spectra& $f(x)$\\
  \midrule
  $S(0,1)=K_1\vee(K_5\cup K_1)$&$[7.66,-0.71,(-1)^4,-2.96]$&$\backslash$ \\
  $S(1,0)=K_2\vee(K_5\cup K_1)$&$[8.47,-0.47,(-1)^5,-3]$&$\backslash$\\
  $S(m,n)=(K_5\cup K_1)\vee(mK_2\cup nK_1)$ $(m,n\ne0)$&$[\partial_1,\partial_2,(-1)^{m+4},(-2)^{n-1},\partial_3,(-3)^m]$&[$x^3-(2n+4m+2)x^2+(2n+8m-28)x+(32m+24n-40)$]\\
  $S(m,0)=(K_5\cup K_1)\vee mK_2$ $(m\ge2)$&$[\partial_1,\partial_2,(-1)^{m+4},(-3)^m]$&[$-x^2+(4m+4)x+20-16m$]\\
  $S(0,n)=(K_5\cup K_1)\vee nK_1$ $(n\ge2)$&$[\partial_1,\partial_2,(-1)^4,\partial_3,(-2)^{n-1}]$&[$x^3-(2n+2)x^2+(2n-28)x+(24n-40)$]\\
  $K_r\vee T_1=K_r\vee(K_4\cup K_1)$&$[\partial_1,\partial_2,(-1)^{r+2},\partial_3](\partial_3>-3)$&[$x^3-(r+2)x^2-(2r+19)x+(3r-16)$]\\
  $K_r\vee T_2=K_r\vee(K_3\cup 2K_1)$&$[\partial_1,\partial_2,(-1)^{r+1},-2,\partial_3](\partial_3>-3)$&[$x^3-(r+3)x^2-(r+24)x+(6r-20)$]\\
  $K_r\vee T_3=K_r\vee(K_3\cup K_1)$&$[\partial_1,\partial_2,(-1)^{r+1},\partial_3](\partial_3>-3)$&[$x^3-(r+1)x^2-(2r+14)x+(2r-12)$]\\
  $K_r\vee T_4(m,n)=K_r\vee(mK_2\cup nK_1)$ $(n,m\ne0)$&$[\partial_1,\partial_2,(-1)^{m+r-1},(-2)^{n-1},\partial_3,(-3)^{m-1}]$&[$x^3\!+\!(6\!-\!2n\!-\!4m\!-\!r)x^2\!+\!(2mr\!-\!8n\!-\!5r\!-\!12m\!+\!nr\!+\!11)x\!-\!(8m\!+\!6n\!+\!6r\!-\!4mr\!-\!3nr\!-\!6)$]\\
  $K_r\vee T_4(m,0)=K_r\vee mK_2$ $(m\ge2)$&$[\partial_1,\partial_2,(-1)^{m+r-1},(-3)^{m-1}]$&[$x^2+(4-r-4m)x+(2mr-4m-3r+3)$]\\
  $K_r\vee T_4(0,n)=K_r\vee nK_1$ $(n\ge2)$&$[\partial_1,\partial_2,(-1)^{r-1},(-2)^{n-1}]$&[$x^2+(3-r-2n)x+(nr-2n-2r+2)$]\\
  $T_1\vee T_1=(K_4\cup K_1)\vee(K_4\cup K_1)$&$[10.71,1,(-1)^6,-2.71,-3]$&$\backslash$\\
  $T_1\vee T_2=(K_4\cup K_1)\vee(K_3\cup 2K_1)$&$[11.32,1.46,(-1)^5,-2,-2.78,-3]$&$\backslash$\\
  $T_1\vee T_3=(K_4\cup K_1)\vee(K_3\cup K_1)$&$[9.65,0.85,(-1)^5,-2.6,-2.9]$&$\backslash$\\
  $T_1\vee T_4(m,n)=(K_4\cup K_1)\vee(mK_2\cup nK_1)$ $(m,n\ne0)$&$[\partial_1,\partial_2,(-1)^{m+3},(-2)^{n-1},\partial_3,\partial_4,(-3)^{m-1}]$&
  [$x^4\!+\!(2\!-\!2n\!-\!4m)x^3\!-\!(6m\!+\!5n\!+\!25)x^2\!+\!(42m\!+\!22n\!-\!98)x\!+\!76m\!+\!57n\!-\!96$]\\
  $T_1\vee T_4(m,0)=(K_4\cup K_1)\vee mK_2$ $(m\ge2)$&$[\partial_1,\partial_2,(-1)^{m+3},\partial_3,(-3)^{m-1}]$&
  [$x^3-4mx^2+(2m-25)x+38m-48$]\\
  $T_1\vee T_4(0,n)=(K_4\cup K_1)\vee nK_1$ $(n\ge2)$&$[\partial_1,\partial_2,(-1)^3,(-2)^{n-1},\partial_3](\partial_3>-3)$&
  [$x^3-(2n+1)x^2+(n-22)x+19n-32$]\\
  $T_2\vee T_2=(K_3\cup 2K_1)\vee(K_3\cup2K_1)$&$[11.87,2,(-1)^4,(-2)^2,-2.87,-3]$&
  $\backslash$\\
  $T_2\vee T_3=(K_3\cup 2K_1)\vee(K_3\cup K_1)$&$[10.34,1.25,(-1)^4,-2,-2.63,-2.95]$&
  $\backslash$\\
  $T_2\vee T_4(m,n)=(K_3\cup 2K_1)\vee(mK_2\cup nK_1)$ $(m,n\ne0$)&$[\partial_1,\partial_2,(-1)^{m+2},(-2)^n,\partial_3,\partial_4,(-3)^{m-1}]$&
  [$x^4\!+\!(1\!-\!2n\!-\!4m)x^3\!-\!(2m\!+\!3n\!+\!34)x^2+(64m\!+\!35n\!-\!124)x\!+\!104m\!+\!78n\!-\!120$]\\
  $T_2\vee T_4(m,0)=(K_3\cup 2K_1)\vee mK_2$ $(m\ge2)$&$[\partial_1,\partial_2,(-1)^{m+2},-2,\partial_3,(-3)^{m-1}]$&
  [$x^3-(4m+1)x^2+(6m-32)x+52m-60$]\\
  $T_2\vee T_4(0,n)=(K_3\cup 2K_1)\vee nK_1$ $(n\ge2)$&$[\partial_1,\partial_2,(-1)^2,(-2)^n,\partial_3](\partial_3>-3)$&
  [$x^3-(2n+2)x^2+(3n-28)x+26n-40$]\\
  $T_3\vee T_3=(K_3\cup K_1)\vee(K_3\cup K_1)$&$[8.57,0.73,(-1)^4,-2.57,-2.73]$&$\backslash$\\
  $T_3\vee T_4(m,n)=(K_3\cup K_1)\vee(mK_2\cup nK_1)$ $(m,n\ne0)$&$[\partial_1,\partial_2,(-1)^{m+2},(-2)^{n-1},\partial_3,\partial_4,(-3)^{m-1}]$&
  [$x^4\!+\!(3\!-\!2n\!-\!4m)x^3\!-\!(8m\!+\!6n\!+\!16)x^2\!+\!(28m\!+\!14n\!-\!72)x\!+\!56m\!+\!42n\!-\!72$]\\
  $T_3\vee T_4(m,0)=(K_3\cup K_1)\vee mK_2$ $(m\ge2)$&$[\partial_1,\partial_2,(-1)^{m+2},\partial_3,(-3)^{m-1}]$&
  [$x^3-(4m-1)x^2-18x+28m-36$]\\
  $T_3\vee T_4(0,n)=(K_3\cup K_1)\vee nK_1$ $(n\ge2)$&$[\partial_1,\partial_2,(-1)^2,(-2)^{n-1},\partial_3](\partial_3>-3)$&
  [$x^3-2nx^2-16x+14n-24$]\\
  \tabincell{l}{$T_4(m_1,n_1)\vee T_4(m_2,n_2)=(m_1K_2\cup n_1K_1)\vee(m_2K_2\cup n_2K_1)$\\($m_1,m_2,n_1,n_2\ne0$)}&$[\partial_1,\partial_2,(-1)^{m_1+m_2},(-2)^{n_1+n_2-2},\partial_3,\partial_4,(-3)^{m_1+m_2-2}]$&
  \tabincell{l}{[$x^4\!+\!(10\!-\!4(m_1\!+\!m_2)\!-\!2(n_1\!+\!n_2))x^3\!+\!(12m_1m_2\!-\!28(m_1\!+\!m_2)\!-\!16(n_1\!+\!n_2)$\\
  $+6(m_1n_2\!+\!m_2n_1)\!+\!37)x^2\!+\!(48m_1m_2\!-\!64(m_1\!+\!m_2)-42(n_1\!+\!n_2)\!+\!30(m_1n_2\!+\!m_2n_1)\!$\\$+\!18n_1n_2\!+\!60)x
  -(48(m_1\!+\!m_2\!-\!m_1m_2)\!+\!36(n_1\!+\!n_2)-36(m_1n_2\!+\!m_2n_1)\!-\!27n_1n_2\!+\!36)$]}\\
  \tabincell{l}{$T_4(m_1,n_1)\vee T_4(m_2,0)=(m_1K_2\cup n_1K_1)\vee m_2K_2$\\$(m_1,n_1\ne0,m_2\ge2)$}&$[\partial_1,\partial_2,(-1)^{m_1+m_2},(-2)^{n_1-1},\partial_3,(-3)^{m_1+m_2-2}]$&
  \tabincell{l}{[$x^3\!+\!(8\!-\!4(m_1\!+\!m_2)\!-\!2n_1)x^2\!+\!(12m_1m_2\!-\!20(m_1\!+\!m_2)\!-\!12n_1\!+\!6m_2n_1\!+\!21)x$\\$-\!24(m_1\!+\!m_2)\!+\!18m_2n_1\!+\!24m_1m_2\!-\!18n_1\!+\!18$]}\\
  \tabincell{l}{$T_4(m_1,n_1)\vee T_4(0,n_2)=(m_1K_2\cup n_1K_1)\vee n_2K_1$\\$(m_1,n_1\ne0,n_2\ge2)$}&$[\partial_1,\partial_2,(-1)^{m_1},(-2)^{n_1+n_2-2},\partial_3,(-3)^{m_1-1}]$&
  \tabincell{l}{[$x^3\!+\!(7\!-\!2(n_1\!+\!n_2)\!-\!4m_1)x^2\!+\!(6m_1n_2\!-\!10(n_1\!+\!n_2)\!-\!16m_1\!+\!3n_1n_2\!+\!16)x$\\$-\!12(n_1\!+\!n_2)\!+\!12m_1n_2\!+\!9n_1n_2\!-\!16m_1\!+\!12$]}\\
  $T_4(m_1,0)\vee T_4(m_2,0)=m_1K_2\vee m_2K_2$ $(m_1,m_2\ge2)$&$[\partial_1,\partial_2,(-1)^{m_1+m_2},(-3)^{m_1+m_2-2}]$&$[x^2\!+\!(6\!-\!4(m_1\!+\!m_2))x\!-\!12(m_1\!+\!m_2)\!+\!12m_1m_2\!+\!9]$\\
  $T_4(m_1,0)\vee T_4(0,n_2)=m_1K_2\vee n_2K_1$ $(m_1,n_2\ge2)$&$[\partial_1,\partial_2,(-1)^{m_1},(-2)^{n_2-1},(-3)^{m_1-1}]$&$[x^2\!+\!(6\!-\!4(m_1\!+\!m_2))x\!-\!12(m_1\!+\!m_2)\!+\!12m_1m_2\!+\!9]$\\
  $T_4(0,n_1)\vee T_4(0,n_2)=n_1K_2\vee n_2K_1$ $(n_1,n_2\ge2)$&$[\partial_1,\partial_2,(-2)^{n_1+n_2-2}]$&$[x^2+(4-2(n_1+n_2))x-4(n_1+n_2)+3n_1n_2+4]$\\
  \bottomrule
 \end{tabular}
\end{sidewaystable}

\end{document}